\documentclass[10pt,reqno]{amsart}

\usepackage{amsmath,amssymb}
\usepackage{a4wide}
\usepackage[all]{xy}
\usepackage{mathbbol}

\newtheorem{prop}{Proposition}[section]
\newtheorem{thm}[prop]{Theorem}
\newtheorem{lem}[prop]{Lemma}
\newtheorem{cor}[prop]{Corollary}

\theoremstyle{remark}
\newtheorem{rem}[prop]{Remark}

\newcommand{\cst}{\ifmmode\mathrm{C}^*\else{$\mathrm{C}^*$}\fi}
\newcommand{\tens}{\otimes}
\newcommand{\atens}{\otimes_{\text{\tiny alg}}}
\newcommand{\eps}{\varepsilon}
\newcommand{\ph}{\varphi}
\newcommand{\comp}{\circ}
\newcommand{\RR}{\mathbb{R}}
\newcommand{\TT}{\mathbb{T}}
\newcommand{\CC}{\mathbb{C}}
\newcommand{\GG}{\mathbb{G}}
\newcommand{\KK}{\mathbb{K}}
\newcommand{\ZZ}{\mathbb{Z}}
\newcommand{\cT}{\mathcal{T}}
\newcommand{\bal}{\boldsymbol{\alpha}}
\newcommand{\id}{\mathrm{id}}
\newcommand{\I}{\mathbb{1}}
\newcommand{\dplus}{\,\dot{+}\,}
\newcommand{\aff}{\,\eta\,}
\newcommand{\Fq}{\mathbb{}F_q}
\newcommand{\cH}{\mathcal{H}}
\newcommand{\ii}{\mathrm{i}}
\newcommand{\cC}{\mathfrak{C}}
\newcommand{\cK}{\mathcal{K}}
\newcommand{\cL}{\mathcal{L}}
\newcommand{\st}{\:\vline\:}
\newcommand{\bga}{\boldsymbol{\gamma}}
\renewcommand{\Bar}[1]{\overline{#1}}
\DeclareMathOperator{\C}{C}
\DeclareMathOperator{\M}{M}
\DeclareMathOperator{\Mor}{Mor}
\DeclareMathOperator{\B}{B}

\DeclareMathOperator{\supp}{supp}
\DeclareMathOperator{\spec}{Sp}
\DeclareMathOperator{\Phase}{Phase}
\DeclareMathOperator{\spn}{span}

\numberwithin{equation}{section}

\begin{document}

\title{Examples of non-compact quantum group actions}

\date{\today}

\author{Piotr M.~So{\l}tan}
\address{Institute of Mathematics, Polish Academy of Sciences\newline
and\newline
Department of Mathematical Methods in Physics, Faculty of Physics, University of Warsaw}
\email{piotr.soltan@fuw.edu.pl}

\thanks{Research partially supported by Polish government grant no.~N201 1770 33.}

\begin{abstract}
We present two examples of actions of non-regular locally compact quantum groups on their homogeneous spaces. The homogeneous spaces are defined in a way specific to these examples, but the definitions we use have the advantage of being expressed in purely $\mathrm{C}^*$-algebraic language. We also discuss continuity of the obtained actions. Finally we describe in detail a general construction of quantum homogeneous spaces obtained as quotients by compact quantum subgroups.
\end{abstract}

%
\maketitle

\section{Introduction}\label{intro}

The main aim of this paper is to present two examples of quantum homogeneous spaces for non-regular quantum groups. We will construct quotients of the quantum ``$az+b$'' and quantum $\mathrm{E}(2)$ groups (from \cite{azb,nazb,pu-soH} and \cite{E2} respectively) by their classical subgroups. In both cases our definition of the quantum homogeneous space will be deeply rooted in the particular form of the \cst-algebras related to these quantum groups. On the other hand, our definitions of quantum homogeneous spaces will be given in purely \cst-algebraic language, without any use of von Neumann algebras. Let us also remark that non-regular quantum groups do not fit into the elaborate and powerful framework of \cite{vaes-imp}. We will prove that in both cases the action of the quantum group on its homogeneous space is continuous in an appropriate sense. The second example (presented in Section \ref{E2sect}) can also be used as basis of a simple definition of a quantum homogeneous space of the form $\GG/\KK$, where $\GG$ is a quantum group and $\KK$ is a \emph{compact} quantum subgroup of $\GG$. We develop this idea in the last section. The notion of a quantum subgroup we will use is more restrictive than that of \cite{vaes-imp}. On the other hand, our construction works for general bisimplifiable Hopf \cst-algebras, not only for locally compact quantum groups. In particular we can use universal versions of quantum groups (\cite{k-univ}, \cite[Section 5]{mmu2}).

We will precede our examples of quantum homogeneous spaces with a brief discussion of actions of classical locally compact groups on \cst-algebras. We shall provide a description of such actions in language appropriate for generalization to quantum groups. Continuity of actions is discussed and a definition of a continuous action of a quantum group is proposed and compared with existing approaches.

The tools used to prove our results will range from the notion of $G$-product and Lanstad algebra to \cst-algebras generated by unbounded affiliated elements. The material on crossed products by of \cst-algebras by actions of abelian groups (including the notion of a $G$-product and Lanstad algebra) can be found in \cite{lanstad,ped} and \cite{kasprzak}. For the notions of multiplier algebra of a \cst-algebra, strict topology, morphisms of \cst-algebras and elements affiliated to \cst-algebras we refer to \cite{unbo,lance,gen}. Let us only mention here that by a \emph{morphism} form a \cst-algebra $A$ to another \cst-algebra $B$ we shall mean a $*$-homomorphism $\ph$ form $A$ to the multiplier algebra $\M(B)$ of $B$ which is \emph{nondegenerate,} i.e.~one whose image contains an approximate unit for $B$. The set of all morphisms from $A$ to $B$ will be denoted by $\Mor(A,B)$. In \cite{gen} the reader will also find a detailed exposition of the concept of a \cst-algebra generated by unbounded elements affiliated with it. We will also use some earlier work on one of our our examples from \cite{pu-so} as well as a wide range of results related to the considered quantum groups (\cite{opeq,E2,contr,azb,nazb,pu-soH}).

The definition of a locally compact quantum group is due to Kustermans and Vaes (\cite{kv}), but we will not be making any use of the Haar weights of the quantum groups under consideration. On the other hand our quantum groups will have a continuous counit. For the theory of compact quantum groups we refer to \cite{cqg}.

Let us briefly describe the contents of the paper. We begin with a section which carefully rephrases the theory of actions of locally compact groups on \cst-algebras int the language suitable for quantum groups (much like various aspects of locally compact group actions on locally compact spaces were treated in \cite{ellwood}). There are no new results in that section, but we feel that some concepts present in the literature require clarification. At the end of this section we define continuous actions of quantum groups on \cst-algebras (quantum spaces) and provide a short explanation of the standard problems encountered in construction of quantum homogeneous spaces. In Section \ref{azbSect} we introduce the quantum ``$az+b$'' groups for various values of the deformation parameter and define the homogeneous space obtained as the quotient by a classical subgroup. This construction was already performed in \cite{pu-so}. Then we analyze the action of the quantum ``$az+b$'' group on its homogeneous space to show that it is continuous. The example from Section \ref{E2sect} is very similar in spirit, but the methods of analyzing this example are significantly different.

Motivated by the Example from Section \ref{E2sect}, in the last section we generalize known construction of a quotient quantum homogeneous space for compact quantum group (cf.~\cite{spheres,podles}) to the situation where the group is no longer compact, but the subgroup is. We show that the action of the original quantum group on the resulting quantum homogeneous space is continuous.

\section{Continuity of classical actions}\label{first}

Let $(B,G,\alpha)$ be a \cst-dynamical system. Then for each $x\in{B}$ we have the function
\begin{equation}\label{balx}
\bal(x):G\ni{t}\longmapsto\alpha_t(x)\in{B}.
\end{equation}
This function has constant norm and is norm-continuous by assumption. In particular it does not vanish at infinity on $G$ unless $x=0$. It follows that $\bal(x)\not\in\C_0(G)\tens{B}$ for non-zero $x$. However, it is easy to see that $\bal$ is a $*$-homomorphism from $B$ to $\M\bigl(\C_0(G)\tens{B}\bigr)$, where the last algebra is naturally identified with the space of all norm-bounded functions $G\to\M(B)$ continuous in the strict topology.

This raises the question how to describe the special property that for each $x\in{B}$ the function $\bal(x)$ has its values in $B$ (not merely in $\M(B)$) and is norm-continuous (not only strictly continuous). The well known answer is that these properties are equivalent to the fact that for each $f\in\C_0(G)$ the element
\begin{equation*}
(f\tens\I)\bal(x)\in\C_0(G)\tens{B}
\end{equation*}
(consider a function $f$ constant and non-zero on a neighborhood of a given point $t\in{G}$).

Let us note here the fact which is well known, but difficult to find in the literature:

\begin{prop}\label{pierwsze}
Let $(B,G,\alpha)$ be a \cst-dynamical system and for each $x\in{B}$ let $\bal(x)$ be defined by \eqref{balx}. Then the linear span of the set
\begin{equation*}
\bigl\{(f\tens\I)\bal(x)\st{f}\in\C_0(G),\;x\in{B}\bigr\}
\end{equation*}
is dense in $\C_0(G)\tens{B}$.
\end{prop}

We will give a very short and easy proof of this fact once appropriate structure on $\C_0(G)$ has been defined (in the proof of Proposition \ref{drugie}). Nevertheless we want to include an elementary proof of Proposition \ref{pierwsze} which does not use the additional structure of $\C_0(G)$.

\begin{proof}[Proof of Proposition \ref{pierwsze}]
We need to approximate any element of $\C_0(G)\tens{B}$ by functions of a specific form. It is enough to approximate elements from a dense subspace $\C_c(G,B)\subset\C_0(G,B)\cong\C_0(G)\tens{B}$. Therefore let $F$ be a fixed continuous function $G\to{B}$ with compact support. For $t\in\supp{F}$ let
\begin{equation*}
U_t=\bigl\{s\in{G}\st\bigl\|F(s)-\alpha_{st^{-1}}\bigl(F(t)\bigr)\bigr\|<\eps\bigr\}.
\end{equation*}
By compactness of $\supp{F}$ there exist $t_1,\ldots,t_n$ such that $\supp{F}\subset\bigcup\limits_{i=1}^nU_{t_i}$. Let $(\chi_i)_{i=1,\ldots,n}$ be a partition of unity subordinated to the covering $(U_{t_i})_{i=1,\ldots,n}$ of $\supp{F}$ and define
\begin{equation*}
X(t)=\sum_{j=1}^n\chi_j(t)\alpha_t\left(\alpha_{t_j^{-1}}\bigl(F(t_j)\bigr)\right).
\end{equation*}
In other words
\begin{equation*}
X=\sum_{j=1}^n(\chi_j\tens\I)\bal\left(\alpha_{t_j^{-1}}\bigl(F(t_j)\bigr)\right).
\end{equation*}

Take now $t\in\supp{F}$. We have
\begin{equation*}
\begin{split}
\bigl\|X(t)-F(t)\bigr\|&=
\biggl\|\sum_{j=1}^n\chi_j(t)\alpha_{tt_j^{-1}}\bigl(F(t_j)\bigr)-F(t)\biggr\|\\
&=\biggl\|\sum_{j=1}^n\chi_j(t)\alpha_{tt_j^{-1}}\bigl(F(t_j)\bigr)-\sum_{j=1}^n\chi_j(t)F(t)\biggr\|\\
&\leq\sum_{j=1}^n\chi_j(t)\bigl\|\alpha_{tt_j^{-1}}\bigl(F(t_j)\bigr)-F(t)\bigr\|.
\end{split}
\end{equation*}
Now the $j$-th term of the last sum is non-zero only if $t\in{U_{t_j}}$. It follows that each term of this sum is either zero or less than $\chi_j(t)\eps$.

We have thus shown that for each $i$ we have $\sup\limits_{t\in{U_{t_i}}}\bigl\|X(t)-F(t)\bigr\|\leq\eps$. On the other hand, outside $\bigcup\limits_{i=1}^nU_{t_i}$ both $X$ and $F$ are zero.
\end{proof}

An immediate corollary of Proposition \ref{pierwsze} is that the linear span of
\begin{equation*}
\bigl\{(f\tens{y})\bal(x)\st{f}\in\C_0(G),\;x,y\in{B}\bigr\}
\end{equation*}
is dense in $\C_0(G)\tens{B}$. Indeed, if $f\tens{y}$ is a simple tensor in $\C_0(G)\tens{B}$ then we can approximate it by finite sums of the form
\begin{equation*}
\sum(f_i\tens\I)\bal(x_i).
\end{equation*}
Now if $(e_\lambda)$ is a (bounded) approximate unit for $B$ then, by the diagonal argument, we can approximate $f\tens{y}$ by sums of the form $\sum(f_i\tens{e_\lambda})\bal(x_i)$ because
\begin{equation*}
\biggl\|\sum(f_i\tens{e_\lambda})\bal(x_i)-g\tens{y}\biggr\|\leq
\biggl\|(\I\tens{e_\lambda})\biggl(\sum(f_i\tens\I)\bal(x_i)-g\tens{y}\biggr)\biggr\|+
\bigl\|g\tens(e_\lambda{y}-y)\bigr\|.
\end{equation*}
Since the approximation works for simple tensors, it also works for general elements of $\C_0(G)\tens{B}$.

\begin{cor}
Let $(B,G,\alpha)$ be a \cst-dynamical system and for each $x\in{B}$ let $\bal(x)$ be defined by \eqref{balx}. Then $\bal\in\Mor\bigl(B,\C_0(G)\tens{B}\bigr)$.
\end{cor}

Let us note here one other density condition appearing in the literature. Quite clearly the linear span of
\begin{equation*}
\biggl\{\int\ph(t)\alpha_t(x)\,dt\st\ph\in{L^1}(G),\;x\in{B}\biggr\}
\end{equation*}
is dense in $B$ (any $x\in{B}$ is the limit of such integrals with $\delta$-like net of integrable functions). In order to rewrite this condition in a more convenient way let us introduce for each $\ph\in{L^1}(G)$ the continuous functional
\begin{equation*}
\omega_\ph:\C_0(G)\ni{f}\longmapsto\int\ph(t)f(t)\,dt.
\end{equation*}
Then we see that the linear span of
\begin{equation*}
\bigl\{(\omega_\ph\tens\id)\bal(x)\st\ph\in{L^1}(G),\;x\in{B}\bigr\}
\end{equation*}
is dense in $B$.

Let us introduce the standard comultiplication $\Delta\in\Mor\bigl(\C_0(G),\C_0(G)\tens\C_0(G)\bigr)$:
\begin{equation}\label{del_class}
\Delta(f)(s,t)=f(st).
\end{equation}
Then one easily checks that $(\Delta\tens\id)\comp\bal=(\id\tens\bal)\comp\bal$.

Clearly one can also see that $(\epsilon\tens\id)\comp\bal=\id$, where $\epsilon$ is the evaluation functional
\begin{equation*}
\C_0(G)\ni{f}\longmapsto{f(e)}\in\CC.
\end{equation*}

We will now state and prove a simple proposition dealing with various conditions defining continuity of a group action in the language adaptable to the more general context of quantum groups.

\begin{prop}\label{drugie}
Let $G$ be a locally compact group and let $B$ be a \cst-algebra. Let
\begin{equation*}
\bal\in\Mor\bigl(B,C_0(G)\tens{B}\bigr)
\end{equation*}
be such that
\begin{itemize}
\item with $\Delta$ defined by \eqref{del_class} we have
\begin{equation}\label{Del-bal}
(\Delta\tens\id)\comp\bal=(\id\tens\bal)\comp\bal
\end{equation}
\item for any $f\in\C_0(G)$ and $x\in{B}$ we have
\begin{equation*}
(f\tens\I)\bal(x)\in\C_0(G)\tens{B},
\end{equation*}
\end{itemize}
Then the following conditions are equivalent
\begin{enumerate}
\item\label{d1} $(\epsilon\tens\id)\comp\bal=\id$, where $\epsilon$ is the evaluation at the neutral element of $G$,
\item\label{d2} $\ker\bal=\{0\}$,
\item\label{d3} the linear span of
\begin{equation}\label{bigspan}
\bigl\{(f\tens\I)\bal(x)\st{f}\in\C_0(G),\;x\in{B}\bigr\}
\end{equation}
is dense in $\C_0(G)\tens{B}$,
\item\label{d4} the linear span of
\begin{equation}\label{smallspan}
\bigl\{(\omega_\ph\tens\id)\bal(x)\st\ph\in{L^1}(G),\;x\in{B}\bigr\}
\end{equation}
is dense in $B$.
\end{enumerate}
Moreover, if the equivalent conditions \eqref{d1}--\eqref{d4} are satisfied then then there exists a continuous action $\alpha$ of $G$ on $B$ such that $\bal$ is defined by \eqref{balx}.
\end{prop}

\begin{proof}
Let us begin with defining a family $(\alpha_t)_{t\in{G}}$ of maps $B\to{B}$ by
\begin{equation*}
\alpha_t(x)=\bal(x)(t)=(\delta_t\tens\id)\bal(x),
\end{equation*}
where $\delta_t$ is the evaluation functional $\C_0(G)\ni{f}\mapsto{f(t)}$. One immediately find that each $\alpha_t$ is an endomorphism of $B$ and that for each $x\in{B}$ the mapping $G\ni{t}\mapsto\alpha_t(x)$ is norm continuous. Moreover, it follows from \eqref{Del-bal} that
\begin{equation*}
\alpha_t\comp\alpha_s=\alpha_{ts}
\end{equation*}
for all $t,s\in{G}$. In particular $\alpha_e$ is an idempotent mapping $B\to{B}$ which commutes with all $\alpha_t$'s. The range of all $\alpha_t$'s is equal to the range of $\alpha_e$ and their kernels are all equal to $\ker{\alpha_e}$.

It follows that \eqref{d1}$\Leftrightarrow$\eqref{d2}. Indeed, \eqref{d1} clearly implies \eqref{d2}, but \eqref{d2} means that for any non-zero $x\in{B}$ the function $t\mapsto\alpha_t(x)$ is non zero. Thus if $x\in\ker{\alpha_e}$ then $\alpha_t(x)=0$ for all $t$ and so $x=0$. In other words, $\alpha_e$ is an idempotent with zero kernel, so $\alpha_e=\id$. This is exactly \eqref{d1}.

In particular if \eqref{d1} is satisfied then all $\alpha_t$'s are automorphisms and $\alpha$ becomes a continuous action of $G$ on $B$. To see that this implies \eqref{d3} without using Proposition \ref{pierwsze} note that the mapping
\begin{equation}\label{can}
\C_0(G)\atens{B}\ni(f\tens{x})\longmapsto(f\tens\I)\bal(x)\in\C_0(G)\tens{B}
\end{equation}
extends to an  automorphism of $\C_0(G)\tens{B}$. Indeed, the map \eqref{can} is bounded and multiplicative (because $\C_0(G)$ is commutative). The inverse mapping is given by
\begin{equation*}
(f\tens{x})\longmapsto(f\tens\I)(\id\tens\kappa)\bigl(\bal(x)\bigr),
\end{equation*}
where $\kappa$ is the automorphism of $\C_0(G)$ given by $\kappa(f)(t)=f(t^{-1})$. Therefore the linear span of \eqref{bigspan} is dense in $\C_0(G)\tens{B}$ as the image of $\C_0(G)\atens{B}$ under an automorphism of $\C_0(G)\tens{B}$.

The fact that \eqref{d3}$\Rightarrow$\eqref{d4} is standard: approximate a simple tensor $f\tens{y}\in\C_0(G)\tens{B}$ by sums of the form
\begin{equation*}
\sum(f_i\tens\I)\bal(x_i).
\end{equation*}
Then take $\ph\in{L^1}(G)$ such that $\omega_\ph(f)=1$. Then the elements
\begin{equation*}
\sum(\omega_{\ph{f_i}}\tens\id)\bal(x_i)
\end{equation*}
approximate $y$ and consequently the linear span of \eqref{smallspan} is dense in $B$.

Finally let us see that \eqref{d4} implies \eqref{d1}. If \eqref{d1} is not satisfied then we know that the ranges of all $\alpha_t$'s are equal to the range of $\alpha_e$ which is strictly contained in $B$ and closed (as the image of a \cst-algebra under a $*$-homomorphism). Then it follows that for any $\ph\in{L^1}(G)$ the element
\begin{equation*}
(\omega_\ph\tens\id)\bal(x)=\int\ph(t)\alpha_t(x)\,dt
\end{equation*}
also lies in the range of $\alpha_e$ and we see that \eqref{d4} is not satisfied.
\end{proof}

In what follows we will deal with quantum groups of the form $\GG=(A,\Delta)$ with continuous counit $\epsilon$. We will say that $\GG$ acts continuously on a \cst-algebra $B$ if there is a morphism $\bal\in\Mor(B,A\tens{B})$ such that
\begin{itemize}
\item $(\Delta\tens\id)\comp\bal=(\id\tens\bal)\comp\bal$,
\item for any $c\in{A}$ and $x\in{B}$ we have $(c\tens\I)\bal(x)\in{A}\tens{B}$,
\item $(\epsilon\tens\id)\comp\bal=\id$.
\end{itemize}

In the literature the last condition is often replaced by either requirement that the linear span of
\begin{equation*}
\bigl\{(c\tens\I)\bal(x)\st{c}\in{A},\;x\in{B}\bigr\}
\end{equation*}
be dense in $A\tens{B}$ or that
\begin{equation*}
\bigl\{(\omega\tens\id)\bal(x)\st\omega\in{A_*},\;x\in{B}\bigr\}
\end{equation*}
be dense in $B$ (where $A_*$ is the space of normal linear functionals on $A$, cf.\cite{mmu2}). In \cite[Proposition 5.8]{bsv} it is shown that these conditions are equivalent for \emph{regular} locally compact quantum groups and that the second condition is strictly weaker for \emph{semi-regular,} but non-regular locally compact quantum groups.

If $\bal\in\Mor(B,A\tens{B})$ is such that $(\Delta\tens\id)\comp\bal=(\id\tens\bal)\comp\bal$ and $\ker{\bal}=\{0\}$ then the established terminology is that $\bal$ is a \emph{reduced} action of $\GG$ on $B$ (see e.g.~\cite{vaes-imp,fischer}). As Proposition \ref{drugie} shows, in case when $\GG$ is a classical locally compact group, all these conditions are equivalent.

%
%

Let $\GG=(A,\Delta)$ and $\KK=(C,\Delta_C)$ be quantum groups. We say that a morphism $\pi\in\Mor(A,C)$ identifies $\KK$ as a \emph{closed quantum subgroup} of $\GG$ if $\pi$ is a surjective $*$-homomorphism $A\to{C}$ and $(\pi\tens\pi)\comp\Delta=\Delta_C\comp\pi$. As we mentioned in Section \ref{intro}, such a definition of a quantum subgroup is more restrictive than that of \cite{vaes-imp}. In the following sections we will give examples of situations where this notion of quantum subgroup applies.

Let us explain the difficulty which we come across as we try to generalize the standard construction of a quotient quantum homogeneous space from \cite{podles}. Let us assume that the \cst-algebras $A$ and $B$ are commutative, so that $A=\C_0(G)$ and $C=\C_0(K)$, where $G$ is a locally compact group and $K$ is a closed subgroup of $G$. In this situation $\pi$ is the restriction morphism $A\ni{f}\mapsto\bigl.f\bigr|_K\in\M(C)$. The range rarely coincides with $C$. In the language of noncommutative topology the homogeneous space $G/K$ is described by the \cst-algebra $B=\C_0(G/K)$. It is a nontrivial question how to identify this \cst-algebra as a subalgebra of $\M(A)=\C_b(G)$. Clearly any element $x\in{B}$ should satisfy
\begin{equation}\label{nax}
(\id\tens\pi)\Delta(x)=x\tens\I,
\end{equation}
but this is certainly not enough. In fact all elements of $\M(B)$ also satisfy this condition. In the examples of Sections \ref{azbSect} and \ref{E2sect} $\GG=(A,\Delta)$ will be a non-regular locally compact quantum group with a subgroup $\KK=(C,\Delta_C)$ (which in both cases will be a classical group) and we will produce a candidate for $B$. It will be a $\cst$-subalgebra of $\M(B)$ consisting of elements $x$ satisfying \eqref{nax} and some additional conditions tailored to these specific examples. The next question which arises naturally after the definition of a homogeneous space is whether the restriction of $\Delta$ to $B$ actually defines an action of $\GG$ on $\GG/\KK$. We will have to check that $\bal=\bigl.\Delta\bigr|_B$ is a morphism from $B$ to $A\tens{B}$ and that algebraic conditions of a (co)action are satisfied. More importantly we would like to show that the action $\bal$ is continuous. None of these facts is totally obvious and we prove them by application of the theory of \cst-algebras generated by unbounded elements (\cite{gen}).

\section{The action of the quantum ``$az+b$'' group on its homogeneous space}\label{azbSect}

The construction of the quantum ``$az+b$'' group starts with choosing the value of the deformation parameter $q$. There are at least three possibilities to choose $q$. They are described in detail in \cite{pu-soH}. One of these possibilities is to take $q\in]0,1[$ (cf.~\cite[Appendix A]{azb}). It is not important for our purposes to dwell on this choice since the aspect of quantum ``$az+b$'' groups we shall need are the same regardless of the value of the deformation parameter. All details can be found in \cite{azb,nazb,pu-soH}.

We let $\Gamma$ be the subgroup of the multiplicative group $\CC\setminus\{0\}$ generated by $q$ and $\{q^{\ii{t}}\st{t}\in\RR\}$ with an appropriate fixed choice of logarithm of $q$. The group $\Gamma$ is then self-dual and we denote by $\chi$ the nondegenerate bicharacter $\Gamma\times\Gamma\to\TT$ such that
\begin{equation*}
\begin{split}
\chi(\gamma,\gamma')&=\chi(\gamma',\gamma),\\
\chi(\gamma,q^{\ii{t}})&=|\gamma|^{\ii{t}},\\
\chi(\gamma,q)&=\Phase{\gamma}
\end{split}
\end{equation*}
for all $\gamma,\gamma'\in\Gamma$ and $t\in\RR$.

Then we let $\Bar{\Gamma}$ be the closure of $\Gamma$ in $\CC$ and let $\beta$ be the action of $\Gamma$ on $\Bar{\Gamma}$ by multiplication. We define $A=\C_0(\Bar{\Gamma})\rtimes_\beta\Gamma$. An affiliated element $b\aff{A}$ is introduced as the image of the canonical generator of $\C_0(\Bar{\Gamma})$ in $A$, while $a\aff{A}$ as the unique affiliated element such that $\spec{a}\subset\Bar{\Gamma}$, $a$ is invertible (with $a^{-1}\aff{A}$) and $\bigl(\chi(a,\gamma)\bigr)_{\gamma\in\Gamma}$ is the canonical family of unitary elements of $\M(A)$ implementing the action $\beta$ on $\C_0(\Bar{\Gamma})$.

By the results of \cite{azb,nazb} (the actual reference depends on the value of $q$, cf.~\cite{pu-soH}) there exists a unique $\Delta\in\Mor(A,A\tens{A})$ such that
\begin{equation*}
\Delta(a)=a\tens{a},\qquad\Delta(b)=a\tens{b}\dplus{b\tens\I}
\end{equation*}
and $\GG=(A,\Delta)$ is a quantum group (in fact a non-regular locally compact quantum group cf.~also \cite{haarslw}).

Let us first see that $\Gamma$ is a subgroup of $\GG$ in the sense described at the end of Section \ref{first}, i.e.~there exists a morphism $\pi\in\Mor\bigl(A,\C_0(\Gamma)\bigr)$ such that
\begin{equation*}
\Delta\comp\pi=(\pi\tens\pi)\comp\Delta
\end{equation*}
and $\pi$ is a surjection onto $\C_0(\Gamma)$. Indeed, since the \cst-algebra $A$ is generated by unbounded elements $a,a^{-1}$ and $b$ affiliated with it (cf. \cite{azb,nazb,pu-soH}), it is enough to define $\pi$ on the generators by putting $\pi(b)=0$ and letting $\pi(a)$ be the canonical generator of $\C_0(\Gamma)$ (the identity function $\Gamma\ni\gamma\mapsto\gamma\in\CC$).

Let us remark here that in \cite{pu-so} it is the dual group of $\Gamma$ which is identified with a subgroup of $\GG$ which we are describing. This is correct since $\Gamma\cong\widehat{\Gamma}$, but the distinction has no bearing on our construction of quantum homogeneous space, so we will stick with the choice of $\Gamma$ as a subgroup of $\GG$.

We shall now define a quantum homogeneous space $\GG/\Gamma$. By definition the \cst-algebra $B$ of continuous functions vanishing at infinity on $\GG/\Gamma$ is the set of those $x\in\M(A)$ for which
\begin{enumerate}
\item\label{lc1} $(\id\tens\pi)\Delta(x)=x\tens\I$,
\item\label{lc2} for any $y\in\cst(\Gamma)\subset\M(A)$ the element $yx\in{A}$,
\item\label{lc3} the mapping $\Gamma\ni\gamma\mapsto\chi(a,\gamma)^*x\chi(a,\gamma)$ is continuous.
\end{enumerate}
Condition \eqref{lc1} selects elements of $\M(A)$ which are ``constant along the fibers'' of the fibration $\GG\to\GG/\Gamma$. However, they cannot correspond to functions vanishing at infinity on $\GG/\Gamma$ (e.g.~the unit $\I$ satisfies this condition). Therefore condition \eqref{lc2} is introduced to select those ``continuous bounded functions on $\GG$'' which not only are constant on cosets of $\Gamma$, but in addition vanish at infinity ``in the direction transversal to the cosets.'' Condition \eqref{lc3} does not have a classical analogue. This definition of the algebra of ``continuous functions vanishing at infinity on $\GG/\Gamma$'' was first introduced in \cite{pu-so}.

It is not difficult to check that $B=\bigl\{f(b)\st{f}\in\C_0(\Bar{\Gamma})\bigr\}$. This is because conditions \eqref{lc1}--\eqref{lc3} are really the Lanstad conditions for the $\Gamma$-product structure on $A=\C_0(\Bar{\Gamma})\rtimes_\beta\Gamma$ (cf.~\cite{lanstad,ped}). Therefore the homogeneous space $\GG/\Gamma$ is a classical space, i.e.~one described by a commutative $\cst$-algebra.

It was pointed out already in \cite{pu-so} that the restriction of $\Delta$ to $B\subset\M(A)$ defines a morphism $\bal\in\Mor(B,A\tens{B})$ and that
\begin{equation*}
(\Delta\tens\id)\comp\bal=(\id\tens\bal)\comp\bal.
\end{equation*}
Moreover one can easily see that $(\epsilon\tens\id)\comp\bal=\id$, where $\epsilon$ is the counit of $\GG$ defined by $\epsilon(a)=1$, $\epsilon(b)=0$. Therefore, in order to say that the action of $\GG$ on the homogeneous space $\GG/\Gamma$ is continuous we only need to prove the following:

\begin{thm}\label{contazb}
Let $\GG=(A,\Delta)$ be the quantum ``$az+b$'' group and let $B\subset\M(A)$ be the $\cst$-algebra of functions vanishing at infinity on the quantum homogeneous space $\GG/\Gamma$. Let $\bal\in\Mor(B,A\tens{B})$ be the morphism defined above. Then for any $c\in{A}$ and any $x\in{B}$ the element
\begin{equation}\label{cbalx}
(c\tens\I)\bal(x)\in{A\tens{B}}.
\end{equation}
\end{thm}

\begin{proof}
We shall identify $B$ with $\C_0(\Bar{\Gamma})$ and $A\tens{B}$ with $A\tens\C_0(\Bar{\Gamma})\cong\C_0(\Bar{\Gamma},A)$. The element $b\aff{B}$ corresponds to the function $\Bar{\Gamma}\ni{z}\mapsto{z}\in\CC$ and $\bal(b)$ to
\begin{equation*}
\Bar{\Gamma}\ni{z}\longmapsto(az\dplus{b})\aff{A}.
\end{equation*}
(cf.~\cite[Formula 2.6]{gen}).

It follows that for any $f\in\C_0(\Bar{\Gamma})$ the image of $f(b)$ under $\bal$ is identified with the element
\begin{equation*}
\Bar{\Gamma}\ni{z}\longmapsto{f}(az\dplus{b})
\end{equation*}
of $\C_b^{\,\text{\tiny strict}}\bigl(\Bar{\Gamma},M(A)\bigr)$. Therefore with $x=f(b)$ the element \eqref{cbalx} corresponds to the function
\begin{equation*}
\Bar{\Gamma}\ni{z}\longmapsto{c}f(az\dplus{b})
\end{equation*}
which is an element of $\C_b(\Bar{\Gamma},A)$. We want to show that this is in fact an element of $\C_0(\Bar{\Gamma},A)=\C_0(\Bar{\Gamma})\tens{A}$.

According to results \cite{opeq,azb,nazb} (again the reference depends on the choice of $q$) there is a continuous function $\Fq:\Bar{\Gamma}\to\TT$ such that $az\dplus{b}=\Fq(z^{-1}ba^{-1})^*az\Fq(z^{-1}ba^{-1})$, so that
\begin{equation}\label{dokladnie}
f(az\dplus{b})=\Fq(z^{-1}ba^{-1})^*f(az)\Fq(z^{-1}ba^{-1}).
\end{equation}
Our aim is to show that for any fixed $c\in{A}$ the norm of 
\begin{equation*}
c\Fq(z^{-1}ba^{-1})^*f(az)\Fq(z^{-1}ba^{-1})
\end{equation*}
tends to $0$ as $z$ goes to $\infty$ of $\Bar{\Gamma}$. Clearly this norm is equal to the norm of
\begin{equation*}
c\Fq(z^{-1}ba^{-1})^*f(az).
\end{equation*}
Since $ba^{-1}\aff{A}$, it is known (see e.g.~\cite[Theorem 5.1]{azb}) that $\Fq(z^{-1}ba^{-1})$ tends to $\I$ strictly in $\M(A)$ az $z\to\infty$ (note that $\Fq(0)=1$). Therefore for any $\eps>0$ there exists a constant $R_\eps$ such that
\begin{equation*}
\bigl\|c\Fq(z^{-1}ba^{-1})^*-c\bigr\|<\eps.
\end{equation*}
for all $z\in\Bar{\Gamma}$ with $|z|>{R_\eps}$. Thus the norm
\begin{equation*}
\bigl\|c\Fq(z^{-1}ba^{-1})^*f(az)\bigr\|\leq
\bigl\|cf(az)\bigr\|+\bigl\|\bigl(c\Fq(z^{-1}ba^{-1})^*-c\bigr)f(az)\bigr\|<
\bigl\|cf(az)\bigr\|+\eps\|f\|.
\end{equation*}
To see that the norm of $cf(az)$ goes to zero as $z\to\infty$ we can assume that $f$ has compact support and that $c$ is of the form
\begin{equation*}
c=\tilde{f}(b)g(a)
\end{equation*}
for some $\tilde{f}\in\C_0(\Bar{\Gamma})$ and $g\in\C_0(\Gamma)$ (the set of such elements is linearly dense in $A$). Then
\begin{equation*}
cf(za)=\tilde{f}(b)G_z(a),
\end{equation*}
where $G_z(z')=g(z')f(zz')$. Since $g$ vanishes at $0$, it is easy to see that $\|G_z\|\xrightarrow[z\to\infty]{}0$ and it follows that $\bigl\|cf(za)\bigr\|<\eps$ for sufficiently large $|z|$.
\end{proof}

\section{The action of the quantum $\mathrm{E}(2)$ group on its homogeneous space}\label{E2sect}

For the description of the quantum $\mathrm{E}(2)$ group we will fix a Hilbert space $\cH$ with an orthonormal basis $(e_{i,j})_{i,j\in\ZZ}$. We will change our notation slightly in order to be in agreement with conventions established in the literature \cite{E2,contr}.
Let $q\in]0,1[$ be a parameter and let $\CC^q$ be the subgroup of $\CC\setminus\{0\}$ generated by $q$ and $\{q^{\ii{t}}\st{t}\in\RR\}$ --- this is $\Gamma$ from Section \ref{azbSect} for the case of real $q$. As before we let $\Bar{\CC^q}$ be the closure of $\CC^q$ in $\CC$.

We will now define two operators on $\cH$. We let $v$ be the unitary operator defined uniquely by $ve_{i,j}=e_{i-1,j}$ for a $i,j\in\ZZ$ and let $n$ be the closed linear operator on $\cH$ such that the linear span of the orthonormal basis $(e_{i,j})_{i,j\in\ZZ}$ is a core of $n$ and $ne_{i,j}=q^ie_{i,j+1}$ for all $i,j\in\ZZ$. Then $n$ is a normal operator with $\spec{n}\subset\Bar{\CC^q}$. We define $A$ as the closure of the set of finite linear combinations
\begin{equation}\label{sumy1}
\sum{f_l}(n)v^l,
\end{equation}
where $f_l\in\C_0(\Bar{\CC^q})$ for all $l\in\ZZ$. With this definition $A$ is a non-degenerate \cst-subalgebra of $\B(\cH)$ isomorphic to $\C_0(\Bar{\CC^q})\rtimes_\beta\ZZ$, where $\beta$ is the natural action by multiplication by $q$ (cf.~\cite{E2}). Moreover $v\in\M(A)$ and $n\aff{A}$ and $A$ is generated by $n$ and $v$ (\cite[Example 4]{gen}).

It is shown in \cite{E2} that there exists a unique $\Delta\in\Mor(A,A\tens{A})$ such that
\begin{equation*}
\Delta(v)=v\tens{v},\qquad\Delta(n)=v\tens{n}\dplus{n}\tens{v^*}
\end{equation*}
and $\GG=(A,\Delta)$ is a non-regular locally compact quantum group (\cite{baaj}) called the \emph{quantum $\mathrm{E}(2)$ group.}

The classical group $\TT$ is a subgroup of $\GG$, i.e.~we have a quantum group morphism $\pi\in\Mor\bigl(A,\C(\TT)\bigr)$ defined uniquely by putting $\pi(n)=0$ and letting $\pi(v)$ be the canonical unitary generator of $\C(\TT)$. We wish to describe the quantum homogeneous space $\GG/\TT$.

To that end let us define structure of a $\ZZ$-product on $A$ different from the standard one (arising from the fact that $A$ is the crossed product of $\C_0(\Bar{\CC^q})$ by $\ZZ$). We define a representation
\begin{equation*}
V:\ZZ\ni{n}\longmapsto{V_n}=v^n\in\M(A)
\end{equation*}
and action of $\TT$ on $A$ by automorphisms $(\tilde{\beta}_\mu)_{\mu\in\TT}$ such that
\begin{equation*}
\tilde{\beta}_\mu(v)=\mu{v},\qquad\tilde{\beta}_\mu(n)=\mu^{-1}n
\end{equation*}
for all $\mu\in\TT$. The triple $(A,V,\tilde{\beta})$ is a $\ZZ$-product. This $\ZZ$-product structure on $A$ is not a ``twist'' of the standard $\ZZ$-product structure on $A$ as defined in \cite{kasprzak}, but let us note that if we denote by $\hat{\beta}$ the action of $\TT$ on $A$ dual to $\beta$ then for each $\mu\in\TT$.
\begin{equation*}
\tilde{\beta}_\mu(y)=U_\mu\hat{\beta}_\mu(y)U_\mu^*,
\end{equation*}
where $U_\mu$ is the unitary operator such that $U_\mu{e}_{i,j}=\mu^{-j}e_{i,j}$.

Following the example from Section \ref{azbSect} we define the \cst-algebra of $B$ of continuous functions vanishing at infinity on $\GG/\TT$ as the set of those $x\in\M(A)$
\begin{enumerate}
\item\label{lc21} $(\id\tens\pi)\Delta(x)=x\tens\id$,
\item\label{lc22} for any $y\in\cst(\ZZ)\subset\M(A)$ the element $yx\in{A}$.
\end{enumerate}

The above conditions are, as in Section \ref{azbSect}, precisely the Lanstad conditions for the $\ZZ$-product structure we have defined (in particular $B$ really is a \cst-algebra). Condition \ref{lc21} is exactly the condition of being $\tilde{\beta}$-invariant, while condition \ref{lc22} says that the ``function'' $x$ vanishes at infinity in the direction transversal to that of the subgroup $\TT$. In fact condition \eqref{lc22} is equivalent to demanding that $x\in{A}$ (because $\cst(\ZZ)$ is unital). The third Lanstad condition is empty because $\ZZ$ is a discrete group.

Before proceeding let us define a certain \cst-algebra which will turn out to be the algebra of continuous unctions vanishing at infinity on $\GG/\TT$. Let $\C_0(\ZZ\cup\{+\infty\})$ be the \cst-algebra of sequences $(x_n)_{n\in\ZZ}$ for which $\lim\limits_{n\to-\infty}x_n=0$ and $\lim\limits_{n\to+\infty}x_n$ exists and is finite. We let $\cC$ be the crossed product $\cC=\C_0(\ZZ\cup\{+\infty\})\rtimes\ZZ$, where the action of $\ZZ$ is by translation. The \cst-algebra $\cC$ is an extension of $\cK$ by $\C(\TT)$. More precisely we have
\begin{equation*}
\cC\cong\biggl\{
\begin{bmatrix}
x&y\\z&u
\end{bmatrix}
\st{x\in\cT},\;y,z,u\in\cK\biggr\}
\end{equation*}
(where $\cT$ is the Toeplitz algebra) and the extension is
\begin{equation*}
\xymatrix{
0\ar[r]&\cK\cong{M_2(\cK)}\ar[r]&\cC\ar[r]^\rho&\C(\TT)\ar[r]&0
}
\end{equation*}
with $\rho$ sending a matrix
$\bigl[
\begin{smallmatrix}
x&y\\z&u
\end{smallmatrix}
\bigr]$
to the symbol of $x$.

\begin{prop} Let $B$ be the algebra of continuous functions vanishing at infinity on $\GG/\TT$. Then
\begin{enumerate}
\item\label{p1} The \cst-algebra $B$ is the closure of the set of finite sums of the form
\begin{equation}\label{sumy}
\sum{f_k}(n)v^k
\end{equation}
where for each $k\in\ZZ$ the function $f_k\in\C_0(\Bar{\Gamma})$ is such that
\begin{equation*}
f_k(\mu{z})=\mu^kf_k(z)
\end{equation*}
for all $z\in\Bar{\Gamma}$ and $\mu\in\TT$.
\item\label{p2} The operator $vn$ is affiliated with $B$,
\item $B$ is generated by $vn$.
\item\label{p4} $B$ is isomorphic to $\cC$.
\end{enumerate}
\end{prop}

\begin{proof}
We shall first prove point \eqref{p1}. Let $B_0$ be the closure of the set of finite sums described in Statement \eqref{p1}. Since each term $f_k(n)v^n$ is invariant under $\tilde{\beta}$ and is already contained in $A$ (as opposed to $\M(A)$) we have $B_0\subset{B}$. Moreover $B_0$ is invariant under conjugation with $v$. Finally let us note that
\begin{equation}\label{dense}
\cst(\ZZ)B_0\cst(\ZZ)
\end{equation}
is dense in $A$ (as before, we identify $\cst(\ZZ)$ with the \cst-subalgebra of $\M(A)$ generated by $v$). This is because the functions $f_k$ appearing in \eqref{sumy} are trigonometric monomials along the individual circles of $\Bar{\CC^q}$. Therefore we can uniformly approximate on each circle sufficiently regular functions (e.g.~twice differentiable) and those are dense in all continuous functions on the circle. This can be done on each circle separately. As a result all elements of the form \eqref{sumy1} lie in the closure of \eqref{dense}. By \cite[Lemma 2.6]{kasprzak} $B_0=B$.

Now we shall deal with points \eqref{p2}--\eqref{p4}. The operator $vn$ is affiliated with $B$. This is because the $z$-transform of $vn$ clearly belongs to $B_0$ and the element $\I-z_{vn}^*z_{vn}=(\I+n^*n)^{-1}$ is a strictly positive element of $B$ (cf.~\cite[Section 1]{gen}). Let us denote by $X$ the operator $vn$ and let $X=U|X|$ be the polar decomposition of $X$. We have $Ue_{i,j}=e_{i-1,j+1}$ and $|X|e_{i,j}=q^ie_{i,j}$. One easily finds that for $f_k\in\C_0(\Bar{\CC^q})$ such that $f_k(\mu{z})=\mu^{k}f_k(z)$ we have
\begin{equation}\label{przezX}
f_k(n)v^k=f_k\bigl(|X|\bigr)U^k.
\end{equation}
A moment of reflection shows now that $B\cong\cC$ and point \eqref{p4} follows.

Let $\rho$ be a representation of $B$ on a Hilbert space $\cL$. Then $\rho(vn)$ is a closed operator on $\cL$. Moreover by \eqref{przezX} the image under $\rho$ of any element $x\in{B_0}$ is determined by the polar decomposition of $\rho(vn)$. This means that $vn$ separates representations of $B$ and it follows that $B$ is generated by $vn$ by \cite[Theorem 3.3]{gen}, since $\bigl(\I+(vn)^*(vn)\bigr)^{-1}\in{B}$.
\end{proof}

Before proceeding let us emphasize the fact that $B\subset{A}$. Moreover since $B\subset{A}$ we have that
\begin{equation*}
\begin{split}
B&=\bigl\{x\in{A}\st\tilde{\beta}_\mu(x)=x\text{ for all }\mu\in\TT\bigr\}\\
&=\bigl\{x\in{A}\st(\id\tens\pi)\Delta(x)=x\tens\I\bigr\}.
\end{split}
\end{equation*}

We also have

\begin{lem}\label{inv}
$A\tens{B}=\bigl\{X\in{A\tens{A}}\st(\id\tens\tilde{\beta}_\mu)(X)=X\;\text{\rm{}for all}\;\mu\in\TT\bigr\}$.
\end{lem}

\begin{proof}
The inclusion ``$\subset$'' is obvious. For the converse one let us take $X\in{A}\tens{A}$ such that $(\id\tens\tilde{\beta}_\mu)(X)=X$ for all $\mu\in\TT$. $X$ is a limit of finite sums of simple tensors:
\begin{equation*}
X=\lim_{k\to\infty}\sum_{i=1}^{N_k}a_i^{(k)}\tens{}b_i^{(k)}.
\end{equation*}
Let $d\mu$ be the normalized Haar measure on $\TT$. We have
\begin{equation*}
X=\int(\id\tens\tilde{\beta}_\mu)(X)\,d\mu
=\lim_{k\to\infty}\sum_{i=1}^{N_k}a_i^{(k)}\tens\int\tilde{\beta}_\mu(b_i^{(k)})\,d\mu
\end{equation*}
Clearly the elements $c_i^{(k)}=\int\tilde{\beta}_\mu(b_i^{(k)})\,d\mu$ are $\tilde{\beta}$-invariant, so that they belong to $B$. It follows that $X\in{A}\tens{B}$.
\end{proof}

Now let us describe the action of $\GG$ on its homogeneous space $\GG/\TT$. The image of $vn\aff{A}$ under $\Delta$ is
\begin{equation*}
v^2\tens{vn}\dplus{vn}\tens\I.
\end{equation*}
By \cite[Theorem 6.1]{wor-nap} and \cite[Theorem 5.1]{opeq} this element is affiliated with $A\tens{B}$. Since $vn$ generates $B$, we see that the map $\bal=\bigl.\Delta\bigr|_B$ is a morphism form $B$ to $A\tens{B}$. Clearly we have
\begin{equation*}
(\Delta\tens\id)\comp\bal=(\id\tens\bal)\comp\bal
\end{equation*}
and $(\epsilon\tens\id)\comp\bal=\id$, where $\epsilon$ is the counit of $\GG$ (defined by $\epsilon(v)=1$ and $\epsilon(n)=0$).

Since $B\subset{A}$ we have $\bal(B)\subset{A}\tens{A}$ by \cite[Formula (57)]{contr}. Note also that
\begin{equation}\label{shift}
(\id\tens\tilde{\beta}_\mu)\comp\Delta=\Delta\comp\tilde{\beta}_\mu.
\end{equation}
In view of Lemma \ref{inv}, this implies that in fact $\bal(x)\in{A}\tens{B}$ for any $x\in{B}$. In particular for any $c\in{A}$ and $x\in{B}$ we have $(c\tens\I)\bal(x)\in{A}\tens{B}$. Thus we obtain

\begin{cor}
$\bal\in\Mor(B,A\tens{B})$ is a continuous action of $\GG$ on $\GG/\TT$.
\end{cor}

\begin{rem}\label{drm}
Let us also note that the set of all elements $(c\tens\I)\bal(x)$ with $c\in{A}$ and $x\in{B}$ is not only contained (as shown above), but also linearly dense in $A\tens{B}$. Indeed if we denote by $E$ the conditional expectation $A\ni{c}\mapsto\int\tilde{\beta}_\mu\,d\mu\in{B}$ used in the proof of Lemma \ref{inv} then it follows from \eqref{shift} that $(\id\tens{E})\comp\Delta=\Delta\comp{E}$. Using the fact that the span of $\bigl\{(a\tens\I)\Delta(b)\st{a,b}\in{A}\bigr\}$ is dense in $A\tens{A}$ we see that
\begin{equation*}
\begin{split}
\mathrm{span}\,&\bigl\{(c\tens\I)\bal(x)\st{c}\in{A},\;x\in{B}\bigr\}\\
&=\mathrm{span}\,\bigl\{(c\tens\I)\Delta(x)\st{c}\in{A},\;x\in{B}\bigr\}\\
&=\mathrm{span}\,\bigl\{(c\tens\I)\Delta\bigl(E(d)\bigr)\st{c,d}\in{A}\bigr\}\\
&=\mathrm{span}\,\bigl\{(c\tens\I)(\id\tens{E})\bigl(\Delta(d)\bigr)\st{c,d}\in{A}\bigr\}\\
&=\mathrm{span}\,\bigl\{(\id\tens{E})\bigl((c\tens\I)\Delta(d)\bigr)\st{c,d}\in{A}\bigr\}
\end{split}
\end{equation*}
is dense in $A\tens{B}$ (cf.~Lemma \ref{inv}).
\end{rem}

\section{Quotients by compact subgroups}

A pair $\GG=(A,\Delta)$ consisting of a \cst-algebra $A$ and a coassociative $\Delta\in\Mor(A,A\tens{A})$ such that
\begin{subequations}\label{simplif}
\begin{align}
\overline{\spn{\bigl\{\Delta(a)(\I\tens{a'})\st{a,a'}\in{A}\bigr\}}}&=A\tens{A},\label{AC}\\
\overline{\spn{\bigl\{(a\tens\I)\Delta(a')\st{a,a'}\in{A}\bigr\}}}&=A\tens{A}.\label{AP}
\end{align}
\end{subequations}
is usually called a bisimplifiable Hopf \cst-algebra (see e.g.~\cite{bs}), while in \cite{mnw} the such objects are called ``proper \cst-bialgebras with cancellation property.'' We will stick to the former terminology.

This section is devoted to the proof of the next theorem which is a direct generalization of the construction in \cite{spheres}[Section 6], \cite[Section 1]{podles} (cf.~also \cite{boca}) to the situation where the original quantum group (or Hopf \cst-algebra) $\GG$ is not compact. We include the assumption that $\GG$ possesses a continuous counit in order to fit the scheme of earlier sections. In case $\GG$ does not have a counit the all statements of Theorem \ref{podlthm} (except \eqref{podlthm3}) are still true.

Note also that many of the results of Section \ref{E2sect} are in fact consequences of Theorem \ref{podlthm}, but the proofs given in Section \ref{E2sect} are different and independent of the next result.

\begin{thm}\label{podlthm}
Let $\GG=(A,\Delta)$ be a bisimplifiable Hopf \cst-algebra with a continuous counit $\epsilon$ and let $\KK=(C,\Delta_C)$ be a compact quantum group. Let $\pi$ be a $*$-homomorphism $A\to{C}$ which is surjective and satisfies
\[
(\pi\tens\pi)\comp\Delta=\Delta_C\comp\pi.
\]
Let $B=\bigl\{x\in{A}\st(\id\tens\pi)\Delta(x)=x\tens\I\bigr\}$. Then
\begin{enumerate}
\item\label{podlthm1} $B$ is a nondegenerate \cst-subalgebra of $A$,
\item\label{podlthm2} the map $\bal=\bigl.\Delta|_B$ has values in $\M(A\tens{B})$ and belongs to $\Mor(A,A\tens{B})$,
\item\label{podlthm3} $(\epsilon\tens\id)\comp\bal=\id$,
\item\label{podlthm4} for any $b\in{B}$ and $a\in{A}$ we have $(a\tens\I)\bal(b)\in{A\tens{B}}$.
\end{enumerate}
\end{thm}

\begin{proof}
Denote $\bga=(\id\tens\pi)\Delta\in\Mor(A,A\tens{C})$. We note the identities
\begin{subequations}
\begin{align}
(\bga\tens\id)\comp\bga&=(\id\tens\Delta_C)\comp\bga,\label{gadeC}\\
(\Delta\tens\id)\comp\bga&=(\id\tens\bga)\comp\Delta\label{dega}
\end{align}
\end{subequations}
which follow directly form the coassociativity of $\Delta$. Next we note that
\begin{subequations}
\begin{align}
\overline{\spn{\bigl\{\bga(a)(\I\tens{c})\st{a}\in{A},\;c\in{C}\bigr\}}}&=A\tens{C},\label{gamC}\\
\overline{\spn{\bigl\{(a\tens\I)\bga(a')\st{a,a'}\in{A}\bigr\}}}&=A\tens{C}\label{gamP}
\end{align}
\end{subequations}
(in particular both sets on the left hand side are contained in $A\tens{C}$). Both \eqref{gamC} and \eqref{gamP} follow from \eqref{simplif} by the surjectivity of $\pi$. If $\KK$ has a continuous counit then it can be shown that $\epsilon=\epsilon_C\comp\pi$ and then $(\id\tens\epsilon_C)\comp\bga=\id$. This means that $\bga$ defines a continuous (right) action of $\KK$ on $A$.

Let $h$ be the Haar measure of $\KK$ and let
\[
E:A\ni{a}\longmapsto(\id\tens{h})\bga(a).
\]
Here we use the strictly continuous extension of $(\id\tens{h})$ to $\M(A\tens{C})$. By \eqref{gamC} $E$ has values in $A$. Indeed we have that $h=vh'$ for some $h'\in{C^*}$ and $v\in{C}$, so that
\[
(\id\tens{h})\bga(a)=(\id\tens{h'})\bigl(\bga(a)(\I\tens{v})\bigr)
\]
is an element of $A$. Moreover $E^2=E$ because
\[
\begin{split}
E^2(a)&=(\id\tens{h}\tens{h})\bigl((\bga\tens\id)\bga(a)\bigr)\\
&=(\id\tens{h}\tens{h})\bigl((\id\tens\Delta_C)\bga(a)\bigr)\\
&=(\id\tens{h})\bga(a)=E(a).
\end{split}
\]
In the above calculation we must remember that we are dealing with extensions of strictly continuous maps to multiplier algebras. For example the identity $(h\tens{h})\comp\Delta_C=h$ follows from the definition of the Haar measure, but $(\id\tens{h}\tens{h})\comp(\id\tens\Delta_C)=(\id\tens{h})$ is an identity between maps on $\M(A\tens{C})$. We obtain it by extending both sides from $A\tens{C}$ to $\M(A\tens{C})$ by strict continuity.

Similarly we show that $E(A)\subset{B}$. Indeed, denoting by $H$ the map $C\ni{c}\mapsto{h(c)\I}$ we have
\begin{equation}\label{gaE}
\begin{split}
\bga\bigl(E(a)\bigr)&=\bga\bigl((\id\tens{h})\bga(a)\bigr)\\
&=(\id\tens\id\tens{h})\bigl((\bga\tens\id)\bga(a)\bigr)\\
&=(\id\tens\id\tens{h})\bigl((\id\tens\Delta_C)\bga(a)\bigr)\\
&=\bigl(\id\tens[(\id\tens{h})\comp\Delta_C]\bigr)\bga(a)\\
&=(\id\tens{H})\bga(a)=E(a)\tens\I.
\end{split}
\end{equation}
As before the identities satisfied by the maps involved in the above calculation are extended by strict continuity to multiplier algebras.

An immediate consequence of \eqref{gaE} is that $B=E(A)$. To prove point \eqref{podlthm1} of Theorem \ref{podlthm} we note that it follows from \eqref{gamP} that
\[
\begin{split}
\spn\bigl\{ab\st{a}\in{A},\;b\in{B}\bigr\}
&=\spn\bigl\{aE(a')\st{a,a'}\in{A}\bigr\}\\
&=\spn\bigl\{a(\id\tens{h})\bga(a')\st{a,a'}\in{A}\bigr\}\\
&=\spn\bigl\{(\id\tens{h})\bigl((a\tens\I)\bga(a')\bigr)\st{a,a'}\in{A}\bigr\}
\end{split}
\]
is a dense subset of $A$. In other words the inclusion of $B$ into $A$ belongs to $\Mor(B,A)$.

In particular $\M(B)$ is the idealizer of $B$ in $\M(A)$ and $\M(A\tens{B})$ is the idealizer of $A\tens{B}$ in $\M(A\tens{B})$ (\cite[Proposition 2.3]{lance}), i.e.
\begin{equation}\label{fromL}
\M(A\tens{B})=\bigl\{X\in\M(A\tens{A})\st{XY},YX\in{A\tens{B}}\ \text{for all}\ Y\in{A\tens{B}}\bigr\}.
\end{equation}

Let us also note that we have
\begin{equation}\label{AtB}
A\tens{B}=\bigl\{Y\in{A\tens{A}}\st(\id\tens\bga)(Y)=Y\tens\I\bigr\}.
\end{equation}
Indeed, the inclusion ``$\subset$'' is clear. For ``$\supset$'' we use the conditional expectation $E$ in the same way as in the proof of Lemma \ref{inv}. Any $Y$ belonging to the right hand side of \eqref{AtB} satisfies $(\id\tens{E})(Y)=Y$ and it is not difficult to see that $A\tens{B}$ is the image of $(\id\tens{E})$. Using \eqref{AtB} we can show that
\begin{equation}\label{MAtB}
\M(A\tens{B})=\bigl\{X\in\M(A\tens{A})\st(\id\tens\bga)(X)=X\tens\I\bigr\}.
\end{equation}
The inclusion ``$\supset$'' follows because if $X\in\M(A\tens{A})$ satisfies $(\id\tens\bga)(X)=X\tens\I$ then for any $Y\in{A\tens{B}}$ we have $XY,YX\in{A\tens{B}}$, so that $X\in\M(A\tens{B})$. Conversely, if $X\in\M(A\tens{B})$ then for any $Y\in{A\tens{B}}$ we have $XY\in{A\tens{B}}$, so by \eqref{AtB} we have
\[
\begin{split}
\bigl((\id\tens\bga)(X)\bigr)(Y\tens\I)
&=\bigl((\id\tens\bga)(X)\bigr)\bigl((\id\tens\bga)(Y)\bigr)\\
&=(\id\tens\bga)(XY)=(XY\tens\I)=(X\tens\I)(Y\tens\I).
\end{split}
\]
This means that $T=(\id\tens\bga)(X)-X\tens\I\in\M(A\tens{A}\tens{B})$ satisfies
\[
T(Y\tens\I)=0
\]
for all $Y\in{A\tens{B}}$. It follows that $T(a\tens{a'}\tens{b})$ is zero for all $a,a'\in{A}$, $b\in{B}$, so that $T=0$ and we obtain ``$\subset$'' in \eqref{MAtB}.

Let us now address point \eqref{podlthm2} of Theorem \ref{podlthm}. The map $\bal=\bigl.\Delta\bigr|_B$ is a composition of the inclusion of $B$ into $A$ with $\Delta$. Therefore it is an element of $\Mor(B,A\tens{A})$. We want to show that it belongs to $\Mor(B,A\tens{B})$. For this it is enough to demonstrate that the range of $\bal$ lies in $\M(A\tens{B})$.

By application of $(\id\tens\id\tens{h})$ to both sides of \eqref{dega} we obtain
\[
\Delta\comp{E}=(\id\tens{E})\comp\Delta.
\]
This shows that the image of $\bal$ is in the image of $(\id\tens{E})$ extended to $\M(A\tens{A})$. Clearly for any $X\in{A\tens{A}}$ and $Y\in{A\tens{B}}$ we have
\begin{equation}\label{XY}
(\id\tens\bga)\bigl[\bigl((\id\tens{E})(X)\bigr)Y\bigr]=XY\tens\I.
\end{equation}
Both sides of this formula are strictly continuous with respect to $X$, so \eqref{XY} holds also for $X\in\M(A\tens{A})$. In view of \eqref{MAtB}, formula \eqref{XY} (together with the analogous one $(\id\tens\bga)\bigl[Y\bigl((\id\tens{E})(X)\bigr)\bigr]=YX\tens\I$ and its extension) for proves that the range of $\bal$ is contained in $\M(A\tens{B})$.

Point \eqref{podlthm3} of Theorem \ref{podlthm} is a straightforward consequence of the definition of $\bal$. To prove point \eqref{podlthm4} we note that
\[
\begin{split}
\bigl\{(a\tens\I)\bal(b)\st{a}\in{A},\;b\in{B}\bigr\}
&=\bigl\{(a\tens\I)\Delta\bigl(E(a')\bigr)\st{a,a'}\in{A}\bigr\}\\
&=\bigl\{(a\tens\I)(\id\tens{E})\Delta(a')\st{a,a'}\in{A}\bigr\}\\
&=\bigl\{(\id\tens{E})\bigl((a\tens\I)\Delta(a')\bigr)\st{a,a'}\in{A}\bigr\}
\end{split}
\]
so that $(a\tens\I)\bal(b)\in{A\tens{B}}$ for all $a\in{A}$, $b\in{B}$.
\end{proof}

\begin{rem}
Note that it follows easily from the last lines of the proof of Theorem \ref{podlthm} that the linear span of the set
\[
\bigl\{(a\tens\I)\bal(b)\st{a}\in{A},\;b\in{B}\bigr\}
\]
is dense in $A\tens{B}$ (cf.~Remark \ref{drm}).
\end{rem}

\end{document}